\documentclass[a4paper,dvipsnames,12pt]{amsart}
\usepackage{etex}
\usepackage{amssymb,amsrefs}
\usepackage[utf8]{inputenc}
\usepackage{soul}
\newcommand{\stkout}[1]{\ifmmode\text{\sout{\ensuremath{#1}}}\else\sout{#1}\fi}
\usepackage{xcolor, tikz}
\usepackage{hyperref}

\newcommand{\F}{\mathcal{F}}

\newenvironment{bibpropia}
  {\bibdiv\biblist\setupbib}
  {\endbiblist\endbibdiv}

\def\setupbib{\catcode`@=\active}
\begingroup\lccode`~=`@
  \lowercase{\endgroup\def~}#1#{\gatherkey{#1}}
\def\gatherkey#1#2{\gatherkeyaux{#1}#2\gatherkeyaux}
\def\gatherkeyaux#1#2,#3\gatherkeyaux{\bib{#2}{#1}{#3}}
\newtheorem{thm}{Theorem}
\newtheorem{prop}{Proposition}

\newtheorem{lemma}{Lemma}
\newtheorem*{conj}{Conjecture}

\theoremstyle{remark}
\newtheorem{defi}{Definition}
\newcommand{\Lip}{\operatorname{Lip}}

\usepackage[normalem]{ulem}

\newcommand{\R}{\mathbb{R}}
\newcommand{\beq}{\begin{equation}}
\newcommand{\eeq}{\end{equation}}
\title[Uncertainty principle and  Vaserstein distance]{An enhanced uncertainty 
principle for the Vaserstein distance}

\author[T. Carroll]{Tom Carroll}

\address{T. Carroll\newline
School of Mathematical Sciences\newline
University College Cork}\email{t.carroll@ucc.ie}

\author[X. Massaneda]{Xavier Massaneda}

\address {X. Massaneda\newline
Departament de Matem\`atiques i Inform\`atica, \newline
Universitat  de Barcelona (UB), BGSMath, 
}
\email{xavier.massaneda@ub.edu}

\author[J. Ortega-Cerd\`a]{Joaquim Ortega-Cerd\`a}

\address{J. 
Ortega-Cerd\`a\newline
Departament de Matem\`atiques i Inform\`atica, \newline
Universitat  de Barcelona (UB), BGSMath
}
\email{jortega@ub.edu}

\thanks{The last two authors have 
been partially supported by the Generalitat de Catalunya 
(grant 2017 SGR 359) and the Spanish Ministerio de Ciencia,  
Innovaci\'on y Universidades (project MTM2017-83499-P)}

\date{\today}


\begin{document}

\begin{abstract}
We improve some recent results of Sagiv and Steinerberger that quantify 
the 
following uncertainty principle: for a function $f$ with mean zero, either 
the size of the zero set of the function or the cost of transporting the mass
of the positive part of $f$ to its negative part must be big. We also provide a 
sharp upper estimate of the transport cost of the positive part of an 
eigenfunction of the Laplacian. This proves a conjecture of Steinerberger and 
provides a lower bound of the size of the nodal set of the eigenfunction. 
\end{abstract}
\maketitle

\section*{Introduction}
\noindent For a continuous function with mean zero, the Vaserstein distance 
between the measures corresponding to the positive and the negative 
parts of the function indicates how oscillatory the function is. 
If this Vaserstein distance is small then the work required to move the 
positive mass to the negative mass is small and so we expect the positive and 
the negative parts of the function to be close together. 
Consequently, we would expect the function to oscillate significantly. 

Our main result is an improvement of  an 
\emph{uncertainty principle\/} due to Sagiv and Steinerberger 
\cite{Steinerberger2} showing that the the zero set of a mean zero, 
continuous  function and the Vaserstein distance between the positive 
and negative parts of the function cannot both be small at the same time.
We prove this result for a function defined in the unit cube of 
$\mathbb R^d$.
It extends to functions defined on a smooth, compact Riemannian manifold 
$M$ of dimension $d$.

Finally, we obtain an upper estimate for this Vaserstein distance in 
the case of 
high frequency eigenfunctions of the Laplacian in $M$ -- by the previous 
uncertainty principle, this indicates that the nodal sets of these 
eigenfunctions should be large. 

A continuous function $f$ on the unit cube $\mathcal Q = [0,1]^d$ in $\R^d$ 
that has zero mean is decomposed into its positive part $f^+ = \max\{f,0\}$ 
and its negative part $f^- =\max\{-f,0\}$. 
The interface between the supports of these two functions is the zero set 
\[
Z(f) = \{ x \in \mathcal Q\,\colon\, f(x) = 0 \}.
\] 
Thinking of $f^+$ as earth that is to be moved and of $-f^-$ as holes that 
need to be filled, then the earth-moving work that is required to fill the 
holes 
is the Vaserstein distance between the measures with densities $f^+$ and $f^-$. 
As mentioned earlier, if the earth mover's distance is small then any earth 
to be moved $f^+$ must be close to a hole that needs to be filled $f^-$, 
and so the interface between the two must be large. 
This is the intuition behind the following quantitative result of 
Steinerberger \cite{Steinerberger0}*{Theorem 2} in dimension 2. 
With a minor abuse of notation, we write $W_1(f^+, f^-)$ for the Vaserstein 
distance between the measures on $\mathcal Q$ with densities 
$f^+$ and $f^-$ respectively relative to Lebesgue measure. 
We write $\mathcal  H^{d-1}(Z(f))$ for the $(d-1)$-dimensional Hausdorff measure 
of the zero set of $f$. 
Then, in dimension $d=2$, 
\beq\label{SBdim2}
W_1(f^+,f^-) \, \mathcal H^1(Z(f)) \, \Vert f\Vert_\infty \gtrsim \Vert f \Vert^2_1.
\eeq
The Vaserstein distance between probability measures $\mu$ and $\nu$ 
on $\mathcal Q$ is defined by 
\beq\label{Vdef}
W_1(\mu,\nu) = \inf_{\rho} \int_{\mathcal Q\times \mathcal Q} 
\vert x-y\vert \,d\rho(x,y)
\eeq 
where the infimum is over all admissible transport plans, that is over all 
probability measures $\rho$ on $\mathcal Q\times \mathcal Q$ with marginals $\mu$ and $\nu$.
Such probability measures $\rho$ are also referred to as 
\textsl{couplings\/} of $\mu$ and $\nu$. 
The monograph \textsl{Optimal Transport, Old and New\/} by Cedric Villani
\cite{Villani} has become a classic reference on optimal transport and 
includes a detailed exposition of the Vaserstein distance, 
also known as the \lq earth-mover's distance\rq. 

The $p$-Vaserstein distance $W_p(\mu,\nu)$ is defined similarly but taking the $p$-norm of $\vert x-y\vert$. 
The $1$-Vaserstein distance has at least two advantages. One is that 
it has an equivalent 
Monge-Kantorovich dual formulation as 
\beq\label{dual}
W_1(\mu,\nu) = \sup_{h\in\Lip_{1,1}(\mathcal Q)} 
	\bigg|\int_{\mathcal Q} h \,d\mu - \int_{\mathcal Q} h \,d\nu\bigg|.
\eeq 
Here $\Lip_{1,1}(\mathcal Q)=\bigl\{h:\mathcal Q\to \mathbb R\ :\ 
|h(x)-h(y)|\leq |x-y|,\ 
x,y\in \mathcal Q\bigr\}$. 

The other, and more important, advantage is that the definition doesn't change if 
in \eqref{Vdef} 
 $d\rho$ is replaced by $d|\rho|$ and $\rho$ is allowed to be a signed measure 
or transport plan 
 on $\mathcal Q\times \mathcal Q$ with marginals $\mu$ and $\nu$ (see \cite{Lev-Quim}). This extra freedom allows 
us to
construct transport plans that lead to better estimates, specifically in the 
course of 
proving Theorem~\ref{Sturm}.

The method of proof that Steinerberger uses to obtain the estimate 
\eqref{SBdim2} does not extend to higher dimensions in any obvious way. 
Using a different method, Sagiv and Steinerberger \cite{Steinerberger2} prove 
that 
\begin{equation*}
W_1(f^+,f^-) \, \mathcal H^{d-1}(Z(f)) \, 
\left(\frac{\|f\|_\infty}{\|f\|_1}\right)^{4-1/d}
	\gtrsim  \Vert f \Vert_1
\end{equation*}
in dimension $d\geq 3$. 
By a modification of the \lq balanced/unbalanced cubes\rq\ method of Sagiv and 
Steinerberger, we can reduce the power from $4-1/d$ to 
$2-1/d$. 

\begin{thm}\label{uncertainty}
Let $f:\mathcal Q \to\R$ be a continuous function with zero mean. 
Let $Z(f)$ be the nodal set $Z(f) = \{x\in \mathcal Q: f(x)= 0\}$. 
Let $\mathcal H^{d-1}(Z(f))$ denote the $(d-1)$-dimensional Hausdorff measure of $Z(f)$. 
Then
\begin{equation}\label{main-estimate}
 W_1(f^+, f^-) \, \mathcal H^{d-1}(Z(f)) \, 
\left(\frac{\|f\|_\infty}{\|f\|_1}\right)^{2-1/d} 
\gtrsim \|f\|_1.
\end{equation}
\end{thm}

The proof is based on a decomposition of the original cube $\mathcal Q$ 
into smaller cubes $Q$ at different scales where either the mass of $|f|$ is 
irrelevant or $\int_Q f^+$ is much larger than $\int_Q f^-$ (or the other way 
around).

This proof extends to a somewhat more 
general setting.
Let $(M, g)$ be a $d$-dimensional, smooth, compact Riemannian manifold without 
boundary and let $dV$ denote the volume form 
associated with $g$.
A function $f:M\to \mathbb R$ 
has zero mean if $\int_M f\, 
dV=0$. 

In this setting, the Vaserstein distance between two probability measures $\mu$ 
and $\nu$ on $M$ is then
\begin{align*}
 W_1(\mu,\nu) &= \inf_{\rho} \int_{M\times M} 
d(x,y) \,d\rho(x,y) \\
&=\sup_{h\in\Lip_{1,1}(M)} \bigg|\int_{M} h \,d\mu - \int_{M} 
h\,d\nu \bigg|,
\end{align*}
where the infimum is over all admissible transport plans $\rho$ from $\mu$ to 
$\nu$. Here $d(x,y)$ stands for the distance induced by the metric $g$ and 
$\Lip_{1,1}(M)=\bigl\{h:M\to \mathbb R\ :\ 
|h(x)-h(y)|\leq d(x,y),\ 
x,y\in M\bigr\}$.

\begin{thm}\label{uncertainty-M}
 Let $(M, g)$ be a smooth, compact 
Riemannian manifold without boundary.
Let $f:M\to \mathbb R$ be a continuous function with zero mean and let 
$Z(f)=\{x\in M : f(x)=0\}$. Then
\begin{equation}\label{main-estimate-M}
 W_1(f^+, f^-) \, \mathcal H^{d-1}(Z(f)) \, 
\left(\frac{\|f\|_{L^\infty(M)}}{\|f\|_{L^1(M)}}\right)^{2-1/d} 
\gtrsim_{(M,g)} \|f\|_{L^1(M)}.
\end{equation}
\end{thm}

We state this result for $M$ compact without boundary because of the 
application we have in mind (see Theorem~\ref{Sturm} below), but it will be clear 
from the proof that the statement holds equally well for $M$ compact with smooth 
boundary.

We also show by means of an example (see Proposition~\ref{example}) that the 
power $2-1/d$ in \eqref{main-estimate} cannot be replaced by any power smaller 
than 
1. In particular, Steinerberger's estimate \eqref{SBdim2} in dimension 2 is 
best possible in this sense. 

The uncertainty principle in Theorem~\ref{uncertainty-M} 
demonstrates that an upper estimate for the Vaserstein distance 
$W_1(f^+,f^-)$ implies a lower estimate on the size of the nodal set.  
In this context, we establish one direction of a conjecture of 
Steinerberger on the Vaserstein distance between the positive and 
negative parts of eigenfunctions of the Laplacian. 
Steinerberger in \cite{Steinerberger1} posed the following conjecture:

\begin{conj} Let $(M, g)$ be a smooth, compact Riemannian manifold without 
boundary. Is it true that if $\phi$ is an $L^2$-normalised 
eigenfunction 
of the Laplacian with eigenvalue $L$, so that $-\Delta \phi = L \phi$ on $(M, 
g)$, 
then
\[
 W_p(\phi^+ ,\phi^-) \simeq_{\,p,\,(M,g)} \frac 1{\sqrt{L\,}} 
	\,\|\phi\|_{L^1(M)}^{1/p}?
\]
\end{conj}

\noindent Steinerberger proves that 
\[
 W_1(\phi^+ , \phi^-) \lesssim_{\,(M,g)} \sqrt{\frac{\log L} {L}} 
 \, \|\phi\|_{L^1(M)}.
\]
We obtain the conjectured upper bound for the case $p=1$ and for all linear 
combinations of eigenfunctions with high frequencies. This formalises the 
intuition that for 
high frequency eigenfunctions it is ``cheap'' to move 
from the positive to the negative part.

\begin{thm}\label{Sturm} Let $(M, g)$ be a smooth, compact 
Riemannian manifold without boundary. Let 
$\{\phi_0,\phi_1,\ldots\}$ be an 
orthonormal basis of $L^2(M)$ consisting of eigenfunctions 
$-\Delta \phi_i = \lambda_i \phi_i$ 
and ordered in such a way that $0=\lambda_0 < \lambda_1\le \lambda_2\le\cdots$. 
Let $f=\sum\limits_{k: \lambda_k\geq L} a_k \phi_k\in L^2(M)$, 
$a_k\in\mathbb R$. Then 
\[
 W_1(f^+ , f^-) \lesssim_{\,(M,g)} \frac 
		{1}{\sqrt{L\,}} \|f\|_{L^1(M)}.
\]
\end{thm}

\noindent The improvement by the factor $\sqrt{\log L}$ 
follows from the construction of a (signed) transport plan that is well 
concentrated on the diagonal.

There is nothing special about the Laplacian in the context of 
Theorem~\ref{Sturm}, 
in that the result holds for any elliptic operator with smooth coefficients in 
the manifold $M$. 
We only need certain estimates on a Bochner-Riesz type kernel that are known to 
hold for general elliptic operators, see \cite{Sogge}. 

Together, Theorem~\ref{uncertainty-M} and Theorem~\ref{Sturm} show 
that when $\phi$ is a linear combination of eigenfunctions of the Laplacian 
with eigenvalues bigger than $L$, 
\[
 \mathcal H^{d-1}(Z(\phi)) \gtrsim \sqrt L 
 \left(\frac{\|\phi\|_{L^1(M)}}{\|\phi\|_{L^\infty(M)}}\right)^{2-1/d}.
 \]
This is a several variables generalization of Sturm's theorem on zeros of 
linear combinations of eigenfunctions, see \cite{Sturmhist}.
 
As such, it goes in the direction of Yau's conjecture that, 
in a smooth compact Riemannian manifold without boundary 
and  for an 
eigenfunction $\phi$ of the Laplacian with eigenvalue $L$, we have $\mathcal H^{d-1} 
(Z(\phi)) \simeq \sqrt{L\,}$. 
The full lower bound in Yau's conjecture, without terms involving 
$L^\infty$ and $L^1$ 
norms of $\phi$, that is $\mathcal H^{d-1} 
(Z(\phi)) \gtrsim \sqrt{L\,}$,  has already been proved by Logunov in 
\cite{Logunov1}.

We finally remark that our method seems to provide information only for 
the Vaserstein distance $W_1$. As mentioned, the definition of $W_1(\mu,\nu)$ 
does not change 
if the transport plan $d\rho$ is replaced by $d|\rho|$, where $\rho$ is a signed 
transport plan.
This fails dramatically for $p>1$.

\begin{prop}
Let $p>1$ and let $\mu, \nu$ be two probability measures in the interval 
$I=[0,1]$. We define
\[
 \widetilde{W}_p^p(\mu, \nu) = \inf_{\rho}\int_{I\times I} 
\vert x-y\vert^p \,d|\rho|(x,y), 
\]
where the infimum is taken over all admissible signed transport plans, that is 
over all 
signed  measures $\rho$ on $I\times I$ with marginals $\mu$ and $\nu$. Then 
$\widetilde{W}_p(\mu, \nu) =0$.
\end{prop}
\begin{proof}
Consider first the case $\mu = \delta_0$ and $\nu = \delta_1$. Then we consider 
the sequence of transport plans $\rho_n$, which consist of $n$ negative Dirac 
deltas and $n+1$ positive Dirac deltas located in points of $I\times I$ as in 
the figure:
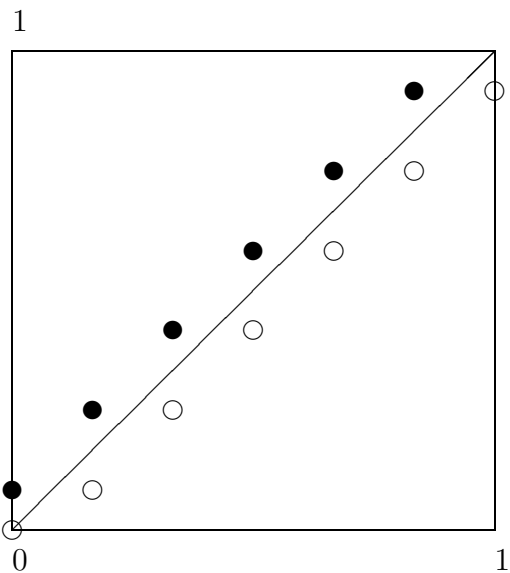
\begin{center}
\setlength{\unitlength}{4144sp}%
\begin{picture}(2976,3298)(2984,-4306)
\put(3045,-4054){\circle{106}}
\put(3045,-3815){\circle*{106}}
\put(3521,-3815){\circle{106}}
\put(3521,-3340){\circle*{106}}
\put(3996,-3340){\circle{106}}
\put(4472,-2864){\circle{106}}
\put(4948,-2388){\circle{106}}
\put(5423,-1912){\circle{106}}
\put(5899,-1437){\circle{106}}
\put(3996,-2864){\circle*{106}}
\put(4472,-2388){\circle*{106}}
\put(4948,-1912){\circle*{106}}
\put(5423,-1437){\circle*{106}}
\put(3045,-4054){\line( 1, 1){2854.500}}
\put(3045,-4054){\framebox(2854,2855){}}
\put(3045,-4291){$0$}
\put(5899,-4291){$1$}
\put(3045,-1080){$1$}
\end{picture}
\end{center}

On the white dots we place a positive Dirac delta and on the black dots a 
negative Dirac delta.  
More precisely we take $\rho_n$ to be
\[
 \rho_n = \delta_{(0,0)} + \sum_{j = 1}^n \delta_{(j/n, 1/(2n)+ (j-1)/n)} - 
\sum_{j=1}^n\delta_{((j-1)/n, 1/(2n)+ (j-1)/n)}.
\]
Clearly the marginals of $\rho_n$ are $\delta_0$ and $\delta_1$.
For any of the Dirac deltas, whether positive or negative and 
located at a point $(x,y)$,  
we have that $|x-y| =  1/(2n)$, except for the Dirac delta at $(0,0)$.
Thus, 
\[
\int_{I\times I} 
\vert x-y\vert^p \,d|\rho_n|(x,y) = \sum_{j=1}^n 2\left(\frac1{2n}\right)^p = 
(2n)^{1-p}. 
\]
Thus,
\[
\widetilde{W}_p^p(\delta_0, \delta_1) \le \liminf_n \int_{I\times I} 
\vert x-y\vert^p \,d|\rho_n|(x,y) = 0. 
\]
This argument can be easily adapted to prove that 
$\widetilde{W}_p(\delta_x, \delta_y) = 0$ for any pair $x,y\in[0,1]$. 
Since linear combinations of Dirac deltas are weak*-dense in the space of 
probability measures, it follows that  
$\widetilde{W}_p(\mu, \nu) = 0$ 
for any probability measures $\mu, \nu$. 
\end{proof}

\emph{Acknowledgements.} We are very thankful to Benjamin Jaye for letting us know that there was a gap in an earlier version of the proof of Theorem~\ref{uncertainty} and for finding the nice fix that he generously lets us use here. The construction of the $Q_j$ in that proof is due to him. 
We are also thankful to Gian Maria Dall'Ara for helpful discussions and to the referee for a careful reading 
of the manuscript and for many thought-provoking suggestions that have resulted in a significant  
improvement of the text.

\section*{Proof of Theorem~\ref{uncertainty}}
Note that in general $f^+\,dV$ and $f^-\,dV$, where $dV$ is Lebesgue measure in $\R^d$, 
are not probability measures, which is the usual setting for the Vaserstein distance. 
However, the distance is well defined for measures with the same total mass. 
Alternatively, notice that the zero mean condition implies that $2 
f^+/\|f\|_1\,dV$ 
and $2 f^-/\|f\|_1\,dV$ are probability measures, so we can define
\[
W_1(f^+, f^-):=\frac{\|f\|_1}2 W_1\Bigl(\frac{2f^+ dV}{\|f\|_1}, 
\frac{2f^- dV}{\|f\|_1}\Bigr).
\]
In any case, replacing $f$ by $f/\|f\|_1$ if necessary, we may assume 
without loss of generality  that $\| f\|_1 =1$ 
and proceed to prove that there is a constant $C_d>0$ such that
\[
 W_1(f^+, f^-) \, \mathcal H^{d-1}(Z(f)) \, \|f\|_\infty^{2-1/d} \ge C_d .
\]
If $\mathcal H^{d-1}(Z(f)) = \infty$ the inequality \eqref{main-estimate} is 
trivially true, 
so we may assume that $\mathcal H^{d-1}(Z(f)) < \infty$.

For convenience, we extend the function $f$ to a function  
defined in all $\R^d$, extending it to be $0$ outside $\mathcal Q$. 
We continue to denote this function by $f$. 
We shall use a decomposition of the cube $\mathcal Q$ into cubes at different scales
defined through a continuous stopping time argument.

The argument draws on constructions used by 
Steinerberger \cite{Steinerberger3} and Sagiv and Steinerberger 
\cite{Steinerberger2}. We need some definitions to describe this decomposition.

For any measurable set $A$ we denote its volume by $V(A)$. 
The side length of a cube $Q$ is denoted by $l(Q)$, 
so  $V(Q) = l(Q)^d$. We write 
\[
V_f^+(Q) = V(Q\cap \{ f>0\}) \mbox{ and } V_f^-(Q) = V(Q\cap \{ f<0\}) 
\]
and note that, since $\mathcal H^{d-1}(Z(f)) < \infty$,  $V(Q\cap \mathcal Q) = 
V_f^+(Q) + V_f^-(Q)$.

\begin{defi}
 We say that a cube $Q$ is \emph{unbalanced} if either 
 \beq\label{unbalanced+}
V_f^+(Q)  \geq 100\times 5^d\,\|f\|_\infty V_f^-(Q)
\eeq
or 
\beq\label{unbalanced-}
V_f^-(Q) \geq 100\times 5^d\, \|f\|_\infty V_f^+(Q).
\eeq
If
\beq\label{balanced}
\frac{1}{100 \times 5^d\, \|f\|_\infty} \leq \frac{V_f^+(Q)}{V_f^-(Q)} \leq 100\times 5^d\, \|f\|_\infty,
\eeq
we say that the cube is \emph{balanced}.
\end{defi}

Since $\int_{\mathcal Q} f^+ = \int_{\mathcal Q}f^- = 1/2$, the cube $\mathcal Q$ is balanced with 
\beq\label{Q0balanced}
\frac{1}{2 \|f\|_\infty} \leq \frac{V_f^+(\mathcal Q)}{V_f^-(\mathcal Q)} \leq 2 \|f\|_\infty.
\eeq

\begin{defi}
 We say that a cube $Q$  is \emph{full\/} whenever 
\[
\int_{Q} |f| \ge \frac{5^{-d}}{10}\, V(Q\cap \mathcal Q).
\]
The \emph{empty\/} cubes are those cubes $Q$ for which 
\[
\int_{Q} |f| < \frac{5^{-d}}{10}\, V(Q\cap \mathcal Q).
\]
\end{defi}

For every $x\in \mathcal Q$ such that $f(x) \ne 0$, there exists $l(x)>0$ such 
that the open cube $Q_x$ centred at $x$ and of side length $l(x)=l(Q_x)$ is simultaneously balanced and 
unbalanced. 
That is, either
\beq\label{perfect+}
V_f^+(Q_x)  = 100\times 5^d\, \|f\|_\infty V_f^-(Q_x)
\eeq
or
\beq\label{perfect-}
V_f^-(Q_x) = 100\times 5^d\, \|f\|_\infty V_f^+(Q_x).
\eeq
This can be achieved by continuity, since for $l$ very small the cube centred at 
$x$ and of side length $l$ is infinitely unbalanced, while for side length $l = 2$
it is balanced, by \eqref{Q0balanced}. Then
there must be an intermediate side length $l(x)$ 
that makes the cube both balanced and unbalanced.

These cubes $Q_x$ cover $\mathcal Q$ (up to at most a zero-measure set).
By the Besicovitch covering theorem \cite{Furedi} , one can find finitely many sequences 
$(x_{i,j})_{i\geq 1}$, $j=1 \ldots 5^d$, such that the cubes $(Q_{x_{i,j}})_{i\geq 1}$ 
are disjoint for each $j$, and together they still cover $\mathcal Q$. That is, 
\[
 \mathcal Q\subset \bigcup_{j = 1}^{5^d}  \bigcup_{i\geq 1} Q_{x_{i,j}}
\]
Since $\int_{\mathcal Q}\vert f\vert =1$, 
there is at least one family of cubes $(Q_{x_{i,j}})_{i\geq 1}$ 
(which, by relabelling, we may assume corresponds to $j=1$) such that 
\[
\sum_{i\geq 1} \int_{Q_{x_{i,1}}} \vert f \vert \ge 5^{-d}.
\]
From this particular sequence of cubes we select those that are full, and 
further relabel the centres of the cubes of this subfamily  as $(x_i)_{i\geq 1}$ and the cubes themselves as 
$Q_{x_i} = Q_i$.
These cubes are disjoint and carry most of the mass.

\begin{prop}\label{prop:1}  There is 
a constant $c>0$ depending only on the dimension $d$ such that
\[
\sum_{i} \int_{{Q_i}} |f| \ge c.
\]
\end{prop}

\begin{proof}
First let us note that the mass of $f$ in 
the cubes ${Q_{x_{i,1}}}$ that are empty cannot be very big:
\[
\sum_{i:\ Q_{x_{i,1}}\,\text{empty}} \int_{ Q_{x_{i,1}} } |f| 
	\le \frac{5^{-d}}{10}\, 
\sum_{i:\ Q_{x_{i,1}}\,\text{empty}}  V( Q_{x_{i,1}} \cap \mathcal Q) \le \frac {5^{-d}}{10}.
\]
Thus the integral over the full cubes satisfies
\[
\sum_{i:\ Q_{x_{i,1}}\,\text{full}} \int_{ Q_{x_{i,1}} } |f|  \ge 5^{-d}\frac 
{9}{10}.\qedhere
\]
\end{proof}

Denote by $\F^+$ the set of indices of the cubes $Q_i$ that 
are full, balanced, and that are unbalanced in the sense that $V_f^+(Q_i)$ dominates $V_f^-(Q_i)$
(\eqref{perfect+} holds). 
Similarly, we denote by $\F^-$ the indices corresponding to those cubes 
$Q_i$  that are full, balanced, and that are unbalanced in the sense that 
$V_f^-(Q_i)$ dominates $V_f^+(Q_i)$ (\eqref{perfect-} holds). 

\begin{lemma}\label{lem:cubes}
For $i \in \F^+$,
each of the following estimates holds:
\begin{equation}\label{minus-plus}
\int_{Q_i} f^- \leq \frac{1}{9} \int_{Q_i} f^+,
\end{equation}
\begin{equation}\label{plus2}
\int_{Q_i} f^+ \geq \frac 9{10}\, \int_{Q_i} \vert f \vert.
\end{equation}
Analogous estimates hold for $i \in \F^-$.
\end{lemma}

\begin{proof} 
If $i\in\F^+$ then, by \eqref{unbalanced+},  
\[
\int_{Q_i} f^- \le \|f\|_\infty V_f^-(Q_i)
\le \frac{5^{-d}}{100}\, V_f^+(Q_i) \le \frac{5^{-d}}{100}\,V(Q_i\cap \mathcal Q).
\]
Since $Q_i$ is full we then have 
\begin{align*}
\int_{Q_i} f^+ & =  \int_{Q_i} |f| -\int_{Q_i} f^-\\
& \geq \frac{5^{-d}}{10}\, V(Q_i\cap \mathcal Q)  - \frac{5^{-d}}{100}\,V(Q_i\cap \mathcal Q)\\
& = \frac{9\times 5^{-d}}{100}\, V(Q_i\cap \mathcal Q).
\end{align*}
These estimates together imply \eqref{minus-plus}.
Finally, 
\[
\int_{Q_i} f^+  = \int_{Q_i} |f| -\int_{Q_i} f^- 
	\geq \int_{Q_i} |f|  - \frac{1}{9} \int_{Q_i} f^+,
\]
which leads to \eqref{plus2}.	
 \end{proof}
 
We are now ready to bound from below both the Hausdorff measure of the 
zero set and the Vaserstein distance between $f^+$ and $f^-$. 
That the Hausdorff measure of the zero set cannot be small 
comes from the fact that the cubes $Q_i$ are balanced.
That the Vaserstein distance between $f^+$ and $f^-$ cannot be small 
comes from the fact that they are unbalanced. 
We first estimate from below the Hausdorff measure of $Z(f)$ 
in $\mathcal Q$. 

\begin{prop}\label{prop:3}
We have:
\begin{equation}\label{area}
\mathcal H^{d-1}(Z(f)\cap \mathcal Q)\gtrsim \frac 1{\|f\|_\infty^{(d-1)/d}} \sum_{i} 
l(Q_i)^{d-1}. 
\end{equation}
\end{prop}

\begin{proof}
We start the proof by considering only the cubes $Q_i$ that are contained in $\mathcal Q$. 
We will deal later with the cubes that intersect the boundary of $\mathcal Q$.

We recall the following relative isoperimetric 
inequality (see \cites{Lions-Facella, Morgan, Ritore}): 
for an open cube $Q$ in $\R^d$ and $K\subset \overline{Q}$,
\begin{equation}\label{isoperimetric}
\mathcal H^{d-1}\big( \partial K\cap Q \big)\gtrsim_d 
	\big(\min\{V(K), V(Q\setminus K)\} \big)^{\frac{d-1}d}.
\end{equation}

Observe that since $Q_i$ is balanced  the volumes in $ Q_i$ 
separated by $Z(f)$ are comparable, up to a factor $\|f\|_\infty$. 
In fact, if 
\[
V_f^-( Q_i) = 100 \times 5^d\,\|f\|_\infty V_f^+( Q_i),
\] 
since  $\|f\|_\infty \geq 1$, we deduce from 
$V( Q_i )=V_f^+( Q_i) + V_f^-( Q_i)$ 
that
\[
V_f^+( Q_i) \geq \frac{V( Q_i)}{(1+100\times 5^d)\, \|f\|_\infty}
\approx\frac{l(Q_i)^d}{\|f\|_\infty}.
\]
Similarly, if $V_f^+( Q_i) = 100 \times 5^d\, \|f\|_\infty 
V_f^-( Q_i)$, 
we find that 
\[
V_f^-( Q_i) \geq \frac{V( Q_i)}{(1+100\times 5^d)\, \|f\|_\infty}
		\approx\frac{l(Q_i)^d}{\|f\|_\infty}.
\]
Then, by the relative isoperimetric inequality \eqref{isoperimetric}, 
\begin{align*}
\mathcal H^{d-1}\big( Z(f)\cap  Q_i \big) & 
	\gtrsim \min\big\{ [V_f^+( Q_i)]^{(d-1)/d}, 
		[V_f^-( Q_i)]^{(d-1)/d} \big\}\\
	& \gtrsim \frac{l(Q_i)^{d-1}}{\|f\|_\infty^{(d-1)/d}}.
\end{align*}
Since the cubes $Q_i$ are disjoint,
\[
\sum_{i\colon Q_i \subset \mathcal Q} \frac{l(Q_i)^{d-1}}{\|f\|_\infty^{(d-1)/d}}
\lesssim
\sum_{i\colon Q_i \subset \mathcal Q}\mathcal  H^{d-1} \big( Z(f)\cap  Q_i \big)
\leq
\mathcal H^{d-1}(Z(f)\cap \mathcal Q).
\]

This last estimate holds only for cubes $Q_i$ that are fully inside $\mathcal Q$. 
There may be others that touch the boundary, but for these we have
\[
\sum_{i \colon Q_i \cap \partial \mathcal Q\ne \emptyset } \frac{l(Q_i)^{d-1}}{\|f\|_\infty^{(d-1)/d}}
	\leq  
\frac{l(\mathcal Q)^{d-1}}{\|f\|_\infty^{(d-1)/d}}
	\lesssim 
\mathcal H^{d-1}(Z(f)\cap \mathcal Q) .
\]
The first inequality holds because the cubes are disjoint and all intersect $\partial \mathcal Q$, and the last one 
because of the relative isoperimetric inequality applied to $\mathcal Q$ (see \eqref{Q0balanced}). 
The estimate \eqref{area} now follows.
\end{proof}

Now we are going to estimate the transport realized in each of the 
full cubes $Q_i$ making use of the fact that they are unbalanced. 

\begin{prop}\label{transport} We have the following estimate of the 
Vaserstein distance between $f^+$ and $f^-$:
\[
 W_1(f^+, f^-) \gtrsim \frac 1{\|f\|_\infty}\sum_{i} 
\frac{\big(\int_{Q_i} |f|\big)^2} {\ l(Q_i)^{d-1}}. 
\]
\end{prop}
\begin{proof}
By definition
\[
 W_1(f^+,f^-) = \inf_{\rho} \int_{\mathcal Q\times \mathcal Q} \vert x-y\vert \,d\rho(x,y),
\]
where $\rho$ is a transport plan between $f^+$ and $f^-$, 
that is $\rho$ is a measure supported on $\mathcal Q\times \mathcal Q$ 
such that for any measurable set $A\subset \mathcal Q$, 
\[
 \int_{A\times \mathcal Q} d\rho(x,y)=\int_A f^+, \qquad \int_{\mathcal Q\times A} 
d\rho(x,y)=\int_A f^-.
\]
We need a uniform lower bound on the transport required for a
general plan $\rho$.
We have,
\begin{align*}
 W_1(f^+,f^-) &\geq \inf_{\rho} \sum_{i} \int_{Q_i\times \mathcal Q} 
 	\vert x-y\vert\, d\rho(x,y)\\
	& \geq \inf_{\rho} \sum_{i} \int_{Q_i\times Q_i^{c}} 
		\vert x-y\vert\, d\rho(x,y)\\
 	&\geq \inf_{\rho} \sum_{i}  \int_{Q_i\times Q_i^{c}} d(x,\partial 
Q_i)\, 
		d\rho(x,y).
\end{align*}
Here, $d(x,\partial Q_i)$ is the distance from $x\in Q_i$ 
to the boundary of the cube $Q_i$.

We now estimate the transport for each $Q_i$.  
Assume $i\in\F^+$, the case $i\in\F^-$ being completely analogous. 
Given any transport plan $\rho$, write
\begin{equation}\label{nu}
 \int_{Q_i\times Q_i^{c}} d(x,\partial Q_i)\, d\rho(x,y)=\int_{Q_i} 
d(x,\partial Q_i)\, d\nu(x),
\end{equation}
where $\nu=\nu_{\rho,i}$ is the measure in $Q_i$ defined by 
$\nu(A)=\rho(A\times Q_i^{c})=\int_{A\times Q_i^{c}}  d\rho(x,y)$, for 
$A\subset Q_i$. By definition $\nu(A)\leq \rho(A\times \mathcal Q)=\int_A f^+$, so 
$\nu \le \chi_{Q_i}f^+dV$. In particular 
\begin{equation}\label{nuup}
\nu(Q_i)\leq \int_{Q_i} f^+.
\end{equation}
On the other hand
\begin{align*}
\nu(Q_i)&= \rho(Q_i \times Q_i^c) = 
\rho(Q_i \times \mathcal Q) - \rho(Q_i \times Q_i) \\
 &=\int_{Q_i} f^+  - \rho(Q_i \times Q_i) .
\end{align*}
Since, by \eqref{minus-plus},
\[
\rho(Q_i \times Q_i) \leq \rho(\mathcal Q\times Q_i)
=\int_{Q_i} f^-\leq \frac 19 \int_{Q_i} f^+
\]
we deduce, using \eqref{plus2}, that  
\begin{equation}\label{nudown}
 \nu(Q_i)\geq \frac 89 \int_{Q_i} f^+ \geq \frac45  \int_{Q_i} \vert f \vert.
\end{equation}
Next, writing the integral in terms of the distribution function,
\begin{align}
\int_{Q_i} d(x, \partial Q_i) \,d\nu(x) 
 	& = \int_0^{l(Q_i)} \nu(\{x\in Q_i : d(x,\partial Q_i) \ge t\})\, dt 
\nonumber\\
	&=l(Q_i)\nu(Q_i) - \int_0^{l(Q_i)} \nu(\{x\in Q_i : d(x,\partial Q_i) < 
t\})\, dt.			\label{tc1}
\end{align}
Since $\nu \le f^+ \chi_{Q_i}\, dV$ and $f^+$ is bounded, we have that
\begin{align*}
 \nu(\{x\in Q_i: d(x,\partial Q_i) < t\})
 	&\le \int f^+ \, \chi_{Q_i\cap \{d(x,\partial Q_i) < t\}} \\
	& \leq \|f\|_\infty V(Q_i\cap \{d(x,\partial Q_i) < t\})\\
	& \leq C\|f\|_\infty\, t\, l(Q_i)^{d-1},
\end{align*}
for some constant $C$ (depending on the dimension). 
Then, by \eqref{nuup},
\[
\nu(\{x\in Q_i: d(x,\partial Q_i) < t\}) 
	\le \min\big\{ \nu(Q_i), C \|f\|_\infty\, t\, l(Q_i)^{d-1} \big\}. 
\]
The crossover point where $\nu(Q_i)$ dominates being when 
\[
t = t_i = \frac{\nu(Q_i)}{C\|f\|_\infty\, l(Q_i)^{d-1}},
\]
we have by \eqref{tc1} that 
\begin{align*}
\int_{Q_i} d(x, \partial Q_i) d\nu(x) 
	&\geq l(Q_i) \nu(Q_i) - \int_0^{t_i} \|f\|_\infty\, t\, l(Q_i)^{d-1} 
\,dt 
		- \int_{t_i}^{l(Q_i)} \nu(Q_i)\,dt\\
	& = \nu(Q_i)\,t_i - \int_0^{t_i} \|f\|_\infty\, t\, l(Q_i)^{d-1} \,dt \\
	& = \frac{1}{2}\frac{\nu(Q_i)^2}{C\|f\|_\infty\, l(Q_i)^{d-1}}.
\end{align*}

Going back to \eqref{nu} and using the estimate \eqref{nudown} 
gives the estimate
\begin{align*}
 \int_{Q_i\times Q_i^{c}} d(x,\partial Q_i)\, d\rho(x,y)&=\int_{Q_i} d(x, 
\partial Q_i) d\nu(x) \\
&\gtrsim\frac{\nu(Q_i)^2}{\|f\|_\infty\, l(Q_i)^{d-1}}\gtrsim 
\frac{\left(\int_{Q_i} 
|f|\right)^2}{\|f\|_\infty \, l(Q_i)^{d-1}}, 
\end{align*}
which finishes the proof of Proposition~\ref{transport}.
\end{proof}

Finally, to conclude the proof of Theorem~\ref{uncertainty}, we use first 
Proposition~\ref{transport} and \eqref{area} to obtain:
\[
 W_1(f^+, f^-)\,\mathcal  H^{d-1}(Z(f)\cap \mathcal Q) \gtrsim \frac 1{\|f\|_\infty^{2-1/d}} 
 	\sum_{i } 
 		\frac{\left(\int_{Q_i} |f|\right)^2}{l(Q_i)^{d-1}} 
	\sum_{i } l(Q_i)^{d-1}
\]
By the Cauchy-Schwarz inequality for sums, applied in the opposite 
direction to usual, and by Proposition~\ref{prop:1} the result follows:
\[
W_1(f^+, f^-)\mathcal  H^{d-1}(Z(f) \cap \mathcal Q) \gtrsim\frac 1{\|f\|_\infty^{2-1/d}}
	\left(\sum_{i}\int_{Q_i} |f| \right)^2 \gtrsim \frac 
1{\|f\|_\infty^{2-1/d}}.
\qedhere
\]

\section*{Proof of Theorem~\ref{uncertainty-M} (Sketch)}

Let $d$ be the dimension of the manifold $M$ and let $\rho$ be the injectivity radius of $(M,g)$,
that is, the supremum of the values $r>0$ such that the exponential map defines a global 
diffeomorphism from the ball with centre 0 and radius $r$ in $\R^d$ onto its image in $M$. 
For $x\in M$ and $r>0$ let $B(x,r)$ denote the ball of centre $x$ and radius $r$ 
in the distance induced by the metric $g$.

Assume, as before, that $\|f\|_{L^1(M)}=1$ and fix $r_0\leq 3\rho$. We start by choosing a ball in $M$ with a substantial part of the $L^1$-norm of $f$: there exist $\epsilon=\epsilon(M)$ and $x_0\in M$ such that 
\[
 \int_{B(x_0,r_0)} |f|\geq \epsilon. 
\]

Denote $\mathcal B= B(x_0,r_0)$.
To adapt the scheme of the previous proof from $\mathcal Q$ to $\mathcal B$ we consider $f$ restricted to $2\mathcal B=B(x_0, 2r_0)$ and extend it outside by 0. We still denote this function by $f$. 

Assume first that $\mathcal B$ is such that
\begin{equation}\label{nine-for-M}
 \frac {\epsilon}{2\| f\|_\infty}\leq \frac{V_f^+(\mathcal B)}{V_f^-(\mathcal B)}\leq \frac{2\| f\|_\infty}{\epsilon}.
\end{equation}
This plays the role of \eqref{Q0balanced} in this proof. The factor $\epsilon$ everywhere is just (the bound of) the $L^1$-norm of $f$ on $ \mathcal B$.

Here we call a ball $B=B(x,r)$  \emph{balanced} if
\[
\frac{\epsilon}{100 \times 5^d\, \|f\|_\infty} \leq \frac{V_f^+(B)}{V_f^-(B)} \leq \frac{100\times 5^d\,\|f\|_\infty}{\epsilon},
\]
and \emph{full} if 
\[
 \int_{ B} |f| \geq \epsilon\, \frac{5^{-d}}{10}\,  V(B\cap 2\mathcal B).
\]

For every $x\in\mathcal B$, let $B_x=B(x,r(x))$ be the ball centered a $x$ and with radius $r(x)$ chosen so that either
\[
V_f^+(B_x)  = \frac{100}{\epsilon}\times 5^d\, \|f\|_\infty V_f^-(B_x)
\]
or
\[
V_f^-(B_x) =  \frac{100}{\epsilon}\times 5^d\, \|f\|_\infty V_f^+(B_x).
\]
Such a radius $r(x)$ exists and is smaller than the injectivity radius, 
because $f$ vanishes outside $2\mathcal B$.

As in the cube case, by the Besicovitch covering theorem there are finitely many families of disjoint balls 
$(B_{x_{i,j}})_{i\geq 1}$ that cover $\mathcal B$. 
We can then select a family, called $(B_{x_{i,1}})_{i\geq 1}$, such that
\[
 \sum_{i\geq 1} \int_{B_{x_{i,1}}} |f|\geq 5^{-d}\int_{\mathcal B} |f|\geq \epsilon\, 5^{-d}.
\]
From these balls we select those that are full, and we relabel them as $(B_i)_{i\geq 1}$. 
With this family of balls, which plays the role of the family $(Q_i)_i$ in the case of the cube, 
we can repeat, mutatis mutandis, the arguments that prove the equivalents of 
Propositions ~\ref{prop:1}, \ref{prop:3} and \ref{transport}, 
and therefore the inequality in Theorem~\ref{uncertainty-M}. 
In the proof of Proposition~\ref{prop:3} we separate the balls $B_i$ inside $2\mathcal B$, 
which are dealt with as before, and those that intersect the boundary of $2\mathcal B$. 
For these ones we use that $\sum_i r_i^{d-1}\lesssim\mathcal H^{d-1}(\partial \mathcal B)$, 
since the centres of the disjoint balls $B_i$ are always in $\mathcal B$.

In case $\mathcal B$ does not satisfy \eqref{nine-for-M} the desired estimate is straightforward. On the one hand, the argument of Proposition~\ref{transport} applied just to the ball $\mathcal B$ yields
\[
 W_1(f^+,f^-)\gtrsim \frac 1{\|f\|_\infty} \frac{\left(\int_{\mathcal B} |f|\right)^2}{r_0^{d-1}}\geq \frac 1{\|f\|_\infty}\frac{\epsilon^2}{r_0^{d-1}}.
\]
On the other hand,
the relative isoperimetric inequality applied to any ball $B(x,r_0)$ with $x\in Z(f)$ yields
\[
 \mathcal H^{d-1}(Z(f)) \gtrsim r_0^{\frac{d-1}d}.
\]
Together these lead to 
\[
  W_1(f^+,f^-)\, \mathcal H^{d-1}(Z(f))\|f\|_\infty \gtrsim_{(M,g)} \|f\|^2_{L^1(M)}.
\]

\section*{An example}
\noindent Next we show that the exponent $2-1/d$ in Theorem~\ref{uncertainty} 
cannot be replaced by any power smaller than 1.
In particular, Steinerberger's uncertainty principle \eqref{SBdim2} in 
dimension~2 is best possible in this sense. 
\begin{prop}\label{example}
Let $\varepsilon>0$. 
There is a continuous function $f_\varepsilon:Q_0\to \R$ such that 
\begin{itemize}
 \item [(i)] $\|f_\varepsilon\|_\infty \simeq \varepsilon^{-1}$ and 
$\|f_\varepsilon\|_1
\simeq 1$,
 \item[(ii)] $Z(f_\varepsilon) = \{x\in [0,1]^d: x_d = 1/2\}$; hence 
$\mathcal  H^{d-1}(Z(f_\varepsilon)) = 1$,
\item[(iii)] $W_1(f^+_\varepsilon, f^-_\varepsilon) \simeq \varepsilon$.
\end{itemize}
Thus, the inequality
\[
 W_1(f^+_\varepsilon, f^-_\varepsilon)\mathcal  H^{d-1}(Z(f_\varepsilon)) 
\left(\frac{\|f_\varepsilon\|_\infty}{\|f_\varepsilon\|_1}\right)^\alpha 
\gtrsim \|f_\varepsilon\|_1,
\] 
does not hold in general for any exponent $\alpha < 1$.
\end{prop}
\begin{proof}
The construction is as follows. 
Write $x \in \R^d$ as $x = (x_{d-1},x_d)$ where $x_{d-1} \in 
\R^{d-1}$. 
Take the function $f_\varepsilon(x) = h_\varepsilon(x_d)$ where the graph of 
$h_\varepsilon$ is as in the picture:

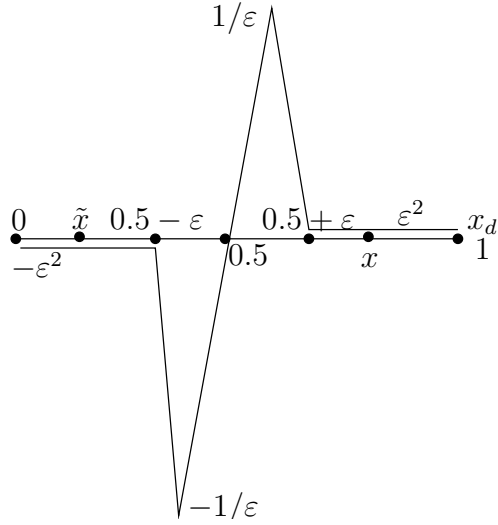
\begin{center} 
\ifx\XFigwidth\undefined\dimen1=0pt\else\dimen1\XFigwidth\fi
\divide\dimen1 by 2945
\ifx\XFigheight\undefined\dimen3=0pt\else\dimen3\XFigheight\fi
\divide\dimen3 by 3137
\ifdim\dimen1=0pt\ifdim\dimen3=0pt\dimen1=4143sp\dimen3\dimen1
  \else\dimen1\dimen3\fi\else\ifdim\dimen3=0pt\dimen3\dimen1\fi\fi
\tikzpicture[x=+\dimen1, y=+\dimen3]
{\ifx\XFigu\undefined\catcode`\@11
\def\temp{\alloc@1\dimen\dimendef\insc@unt}\temp\XFigu\catcode`\@12\fi}
\XFigu4143sp
\ifdim\XFigu<0pt\XFigu-\XFigu\fi
\clip(4800,-7440) rectangle (7745,-4303);
\tikzset{inner sep=+0pt, outer sep=+0pt}
\pgfsetlinewidth{+7.5\XFigu}
\filldraw  (4843,-5772) circle [radius=+28];
\filldraw  (6082,-5772) circle [radius=+28];
\filldraw  (7460,-5772) circle [radius=+28];
\filldraw  (5669,-5772) circle [radius=+28];
\filldraw  (6578,-5772) circle [radius=+28];
\filldraw  (5220,-5760) circle [radius=+28];
\filldraw  (6930,-5760) circle [radius=+28];
\draw 
(4870,-5827)--(5669,-5827)--(5807,-7425)--(6358,-4395)--(6578,-5717)--(7460,
-5717);
\draw (4843,-5772)--(7460,-5772);
\pgftext[base,left,at=\pgfqpointxy{6100}{-5930}] {$0.5$}
\pgftext[base,left,at=\pgfqpointxy{4815}{-5717}] {$0$}
\pgftext[base,left,at=\pgfqpointxy{7102}{-5690}] {$\varepsilon^2$}
\pgftext[base,left,at=\pgfqpointxy{6000}{-4500}] {$1/\varepsilon$}
\pgftext[base,left,at=\pgfqpointxy{5862}{-7425}] {$-1/\varepsilon$}
\pgftext[base,left,at=\pgfqpointxy{7560}{-5895}] {$1$}
\pgftext[base,left,at=\pgfqpointxy{4815}{-6000}] {$-\varepsilon^2$}
\pgftext[base,left,at=\pgfqpointxy{6300}{-5715}] {$0.5+\varepsilon$}
\pgftext[base,left,at=\pgfqpointxy{5400}{-5715}] {$0.5-\varepsilon$}
\pgftext[base,left,at=\pgfqpointxy{6885}{-5940}] {$x$}
\pgftext[base,left,at=\pgfqpointxy{7515}{-5715}] {$x_d$}
\pgftext[base,left,at=\pgfqpointxy{5175}{-5715}] {$\tilde x$}
\endtikzpicture
\end{center}

Properties (i) and (ii) of the function $f_\varepsilon$ are then immediate. 

The function $h_\varepsilon$ is symmetric about $x_d=0.5$, so that 
$h_\varepsilon(1-x_d)=-h_\varepsilon (x_d)$, 
$x_d\in [0,1]$. 
For $x = (x_{d-1},x_d) \in Q_0$ we write $\widetilde x =\widetilde x(x)$ 
for the point 
$(x_{d-1},1-x_d) \in Q_0$, the reflection of $x$ in the hyperplane $x_d = 1/2$.
Observe that $|x-\widetilde x|=|1-2x_d|$.
Then, $f_\varepsilon^+(\widetilde x) = f_\varepsilon^-(x)$.

To prove the upper bound in (iii), consider the following transport plan
\[
 \rho(x,y) = f_\varepsilon^+(x)\,\delta_{\widetilde x}(y).
\]
Notice that it has the correct marginals:
\begin{align*}
\int_{y\in Q_0} d\rho(x,y) 
	& =   f_\varepsilon^+(x) \int_{y\in Q_0} \delta_{\widetilde 
x}(y) 
		=f_\varepsilon^+(x), \\
\int_{x\in Q_0} d\rho(x,y) 
	& = \int_{x\in Q_0} f_\varepsilon^+(x)\,\delta_{\widetilde 
x}(y)
		 =f_\varepsilon^+(\widetilde y) = f_\varepsilon^-(y).
\end{align*}
Therefore,
\begin{align*}
W_1(f^+_\varepsilon, f^-_\varepsilon) 
	& \le \iint_{Q_0\times Q_0} \vert x-y\vert \, d\rho(x,y)\\
	&  = \int_{x \in Q_0} f_\varepsilon^+(x) \int_{y\in Q_0} \vert 
x-y\vert\,  
				\delta_{\widetilde x}(y)\\
	&  = \int_{x \in Q_0}  f_\varepsilon^+(x) \, \vert x-\widetilde 
x\vert \,dV(x)\\
	&  = 2\int_{x \in Q_0}  \big(x_d - \tfrac12\big)\, 
f_\varepsilon^+(x) \, dV(x)
		\lesssim \varepsilon.
	\end{align*}

For the lower bound we use the Monge-Kantorovich duality lemma 
(see \eqref{dual} or \cite{Villani}*{Formula~(6.3)}): 
\[
 W_1(\mu, \nu)=\sup_{g\in\Lip_{1,1}(Q_0)} \left|\int_{Q_0} g \, 
(d\mu-d\nu)\right|.
\] 
Taking $g(x) = x_d - \tfrac12$ we have
\begin{align*}
W_1(f^+_\varepsilon, f^-_\varepsilon)& \geq  
\int_{Q_0} (x_d - \tfrac12) \, 
(f^+_\varepsilon(x)-f^-_{\varepsilon}(x))\, dV(x)\\
& = \int_{Q_0} (x_d - \tfrac12)\, f_\varepsilon (x)\, dV(x)  
 \gtrsim \varepsilon. \qedhere
\end{align*}
\end{proof}
With a similar example one can check that in dimension $2$, 
the inequality \eqref{SBdim2}, that is 
$W_1(f^+,f^-) \, \mathcal  H^1(Z(f)) \, \Vert f\Vert_\infty \geq C \Vert f \Vert^2_1$, 
cannot hold with a constant $C$ greater than $1$.
It is an interesting problem to determine the best constant in the
equality \eqref{SBdim2}.

\section*{Proof of Theorem~\ref{Sturm} on eigenfunctions of the Laplacian}

Let $d$ be the dimension of $M$ and denote by $V$ the volume form  on $M$ associated to $g$ and 
normalised so that $V(M) = 1$. 

In order to construct a transport plan between $f^+$ and $f^-$ we 
consider an auxiliary kernel. 
Let $a:[0,1]\to \R$ be a smooth decreasing function such that 
$a(t)\equiv 1$ in $[0,1/4]$ and $a(t)\equiv 0$ in $[3/4,1]$. 

Observe that $\phi_0(x)=1$ and  therefore
\begin{equation}\label{orthogonal}
 \int_M \phi_i(x)\, dV(x)=\langle\phi_i, 
\phi_0\rangle=0,
 \quad  i \geq 1.
\end{equation}

For any $L >0$, we write 
\[
 B_L(x,y) = \sum_{\lambda_i < L}  a(\lambda_i/L) \,
\phi_i(x) \, \phi_i(y),\ x,\, y \in M.
\]
This is a kernel of Bochner-Riesz type. It is a smoothed out version of the 
Bergman kernel that gives the orthogonal projection from $L^2(M)$ to the span 
generated by the first eigenvector of the Laplacian, in the same spirit as the 
Riesz kernels are a smoothed version of the Dirichlet kernel on
trigonometric sums. See \cites{Sogge, Stein} for the basic properties of the 
kernel.

It is proved in \cite{Sogge}*{Lemma 
2.1} that the 
following pointwise estimates hold: for any $N>0$ there exists $C_N>0$ such that
\begin{equation}\label{estimate}
|B_L(x,y)|\le C_{N} \, \frac{L^{d/2}}{\big[ 1+\sqrt{L\,}d(x,y)\big]^N},\qquad 
x,y\in M
\end{equation}

Now we use a slightly different definition of the Vaserstein distance 
(see \cite{Lev-Quim}*{Formula (43)}):
\[
 W_1(\mu,\nu)=\inf_\rho \iint_{M\times M} d(x,y)\, d|\rho|(x,y),
\]
where $\rho$ are now \emph{signed} measures on $M\times M$ with marginals 
$\rho(\cdot, 
M)=\mu$, $\rho(M,\cdot)=\nu$.
This follows from the 
estimate of the Vaserstein distance using the dual expression 
\eqref{dual}:
\[
 W_1(\mu,\nu) = \sup_{h\in\Lip_{1,1}(M)}  \left|\int_M h(w) 
\Bigl(d\mu(w)\,dV(w) - d\nu(w)\,dV(w)\Bigr)\right|.
\]
A direct estimate yields, for any signed measure $\rho$ with marginals $\mu$ 
and 
$\nu$,
\begin{align*}
 W_1(\mu,\nu) &=\sup_{h\in\Lip_{1,1}(M)}\left|\int_M h(w)
 	\Bigl[\int_{y\in M} d\rho(w,y) - 
		\int_{x\in M} d\rho(x,w)\Bigr]\right|\\
&\le \sup_{h\in\Lip_{1,1}(M)} \int_{M}\int_M |h(x)-h(y)|\, d|\rho|(x,y)\\
& \le \iint_{M\times M} d(x,y)\, d|\rho|(x,y). 
\end{align*}
The other inequality is trivial.

Let $\sigma$ be the pushforward of the measure 
$f^-\, dV$ by the diagonal map $F: M\to M\times M$ defined as 
$F(x) = (x,x)$, that is  $\sigma = F_*(f^-dV)$. 
The measure $\sigma$ is supported on the diagonal 
$ \mathcal{D}=\{(x,y)\in M\times M :x=y\}$.
Define a signed measure on $M\times M$ by
\[
 \rho_L(x,y) = B_L(x,y)\, f(x)\, dV(x)\, dV(y) + \sigma(x,y).
\]
We compute  the marginals of  $\rho_L$. It is 
straightforward that both marginals of $\sigma$ are $f^-dV$, 
so we are left with the computation of the marginals of the first term in 
$\rho_L$.
Clearly
\[
\int_{y\in M} B_L(x,y) \, f(x)\,dV(x)\,dV(y)=
f(x)\,dV(x)\int_{M} B_L(x,y) \,dV(y)
\]
and, by definition and by \eqref{orthogonal},
\begin{align*}
\int_{M} B_L(x,y) \,dV(y) 
&= \sum_{\lambda_i < L}  a\bigg(\frac{\lambda_i}L\bigg)  \phi_i(x) 
\int_M \phi_i(y)\, dV(y)\\
&=\phi_0(x)\, V(M)=1. 
\end{align*}
Hence, the marginal of the first term in $\rho_L$ with respect to $y\in M$ is 
$f(x)\, dV(x)$, and therefore
\[
 \int_{y\in M} d\rho_L(x,y)=f(x)\, dV(x)+ f^-(x)\, 
dV(x)=f^+(x)\, 
dV(x).
\]

For the other marginal we use the orthogonality of $f$ to all 
$\phi_i$, $\lambda_i< L$, 
(since it is a linear combination of eigenfunctions of $-\Delta$ with 
eigenvalues $\lambda_k\geq L$).
Thus,
\begin{align*}
\int_{x\in M} B_L(x,y)\, &f(x)\,dV(x)\,dV(y)\\
& =
\sum_{\lambda_i < L} a\biggl(\frac{\lambda_i}L\biggr) \phi_i(y)\, dV(y)
\int_M \phi_i(x)\,f(x) \, dV(x) = 0,
\end{align*}
and the second marginal of $\rho_L$ reduces to that of $\sigma$, which 
is $f^-(y)\, dV(y)$.

Now that we have checked that $\rho_L$ has the correct marginals let us prove 
the inequality in the statement of Theorem~\ref{Sturm}.

Since $\sigma$ is supported on the diagonal, it does not contribute to this 
last 
integral. Using \eqref{estimate}, we are led to:
\begin{align*}
 W_1(f^+,f^-) &\lesssim \int_M 
\int_M|f(x)|\,\frac{L^{d/2} d(x, y)}{\big[1+\sqrt{L\,}d(x,y)\big]^N}\, 
dV(x)\, dV(y)\\
&\leq
\frac{\|f\|_1}{\sqrt{L\,}} \sup_{x\in M}\int_M 
\frac{L^{d/2}\sqrt{L\,}d(x,y)}{\big[1+\sqrt{L\,}d(x,y)\big]^N} \, dV(y)\\
&\leq \frac{\|f\|_1}{\sqrt{L\,}} \sup_{x\in M}\int_M 
\frac{L^{d/2}}{\big[1+\sqrt{L\,}d(x,y)\big]^{N-1}} \, dV(y).
\end{align*}
We are still free to choose $N$. 
We pick $N > d+1$ (the choice $N=d+2$ works fine) and complete the proof of 
Theorem~\ref{Sturm} by showing that
there is a finite constant $C$ independent of $L$ such that 
\begin{equation}\label{N}
\sup_{x\in M}\int_M \frac{ L^{d/2} }{\big[ 1+\sqrt{L\,}d(x,y) 
\big]^{N-1}}\,dV(y) \leq C.
\end{equation}
Writing the integral in terms of the distribution function and substituting 
$t=\big(1+\sqrt{L\,} 
s\big)^{-N+1}$ 
we obtain 
\begin{align*}
\int_M &\frac{L^{d/2}}{\big[1+\sqrt{L\,} d(x,y)\big]^{N-1}}\,dV(y)\\
&\hskip1cm =L^{d/2} \int_0^1 V\left(\left\{ 
	y \,\colon\,\big[ 1+\sqrt{L\,}\, d(x,y) \big]^{-N+1}>t   
\right\}\right) 
\,dt\\
&\hskip1cm=(N-1) L^{d/2} \int_0^\infty V\big(\{y : d(x,y)<s\}\big) 
	\frac{\sqrt{L\,} ds}{\big(1+\sqrt{L\,} s\big)^N}.
\end{align*}
Since $M$ is compact, 
the volume of a geodesic ball $\{y : d(x,y)<s\}$  
is at most a (global) constant times $s^d$. We deduce, finally, that
\begin{align*}
\int_M \frac{L^{d/2}}{\big[1+\sqrt{L\,} d(x,y)\big]^{N-1}}\,dV(y) 
& \lesssim L^{d/2} \int_0^\infty s^d \frac{\sqrt{L\,} ds}{\big(1+\sqrt{L\,} 
s\big)^N} \\
& =\int_0^\infty \frac {u^d\, du}{(1+u)^N}\ \lesssim 1, 
\end{align*}
which proves \eqref{N} and completes the proof of Theorem~\ref{Sturm}.

\nocite{*}
\begin{bibpropia}

@article{Sturmhist,
    AUTHOR = {B\'{e}rard, Pierre}, 
    AUTHOR = {Helffer, Bernard},
     TITLE = {Sturm's theorem on zeros of linear combinations of
              eigenfunctions},
   JOURNAL = {Expo. Math.},
    VOLUME = {38},
      YEAR = {2020},
    NUMBER = {1},
     PAGES = {27--50},
      ISSN = {0723-0869},
       DOI = {10.1016/j.exmath.2018.10.002},
       URL = {https://doi.org/10.1016/j.exmath.2018.10.002},
}

@article{Furedi,
    AUTHOR = {F\"{u}redi, Zolt\'{a}n},  
    AUTHOR = {Loeb, Peter A.},
     TITLE = {On the best constant for the {B}esicovitch covering theorem},
   JOURNAL = {Proc. Amer. Math. Soc.},
    VOLUME = {121},
      YEAR = {1994},
    NUMBER = {4},
     PAGES = {1063--1073},
      ISSN = {0002-9939},
       DOI = {10.2307/2161215},
       URL = {https://doi.org/10.2307/2161215},
}

@article{Gariboldi,
author = {Gariboldi, Bianca},
author = {Gigante, Giacomo},
year = {2018},
title = {{Optimal asymptotic bounds for 
designs on manifolds}},
note = {\url{http://arxiv.org/abs/1811.12676}},
}

@article{Logunov1,
    AUTHOR = {Logunov, Alexander},
     TITLE = {Nodal sets of {L}aplace eigenfunctions: proof of
              {N}adirashvili's conjecture and of the lower bound in {Y}au's
              conjecture},
   JOURNAL = {Ann. of Math. (2)},
    VOLUME = {187},
      YEAR = {2018},
    NUMBER = {1},
     PAGES = {241--262},
      ISSN = {0003-486X},
       DOI = {10.4007/annals.2018.187.1.5},
       URL = {https://doi.org/10.4007/annals.2018.187.1.5},
}

@article{Lev-Quim,
    AUTHOR = {Lev, Nir},
    AUTHOR = {Ortega-Cerd\`a, Joaquim},
     TITLE = {Equidistribution estimates for {F}ekete points on complex
              manifolds},
   JOURNAL = {J. Eur. Math. Soc. (JEMS)},
    VOLUME = {18},
      YEAR = {2016},
    NUMBER = {2},
     PAGES = {425--464},
      ISSN = {1435-9855},
       DOI = {10.4171/JEMS/594},
       URL = {https://doi.org/10.4171/JEMS/594},
}

@article{Lions-Facella,
    AUTHOR = {Lions, Pierre-Louis},
    AUTHOR = {Pacella, Filomena},
     TITLE = {Isoperimetric inequalities for convex cones},
   JOURNAL = {Proc. Amer. Math. Soc.},
    VOLUME = {109},
      YEAR = {1990},
    NUMBER = {2},
     PAGES = {477--485},
      ISSN = {0002-9939},
       DOI = {10.2307/2048011},
       URL = {https://doi-org.ucc.idm.oclc.org/10.2307/2048011},
}

@book{Morgan,
    AUTHOR = {Morgan, Frank},
     TITLE = {Geometric measure theory},
   EDITION = {Fifth Edition},
      NOTE = {A beginner's guide,
              Illustrated by James F. Bredt},
 PUBLISHER = {Elsevier/Academic Press, Amsterdam},
      YEAR = {2016},
     PAGES = {viii+263},
      ISBN = {978-0-12-804489-6},
}

@article{Ritore,
    AUTHOR = {Ritor\'{e}, Manuel},
    author = {Vernadakis, Efstratios},
     TITLE = {Isoperimetric inequalities in {E}uclidean convex bodies},
   JOURNAL = {Trans. Amer. Math. Soc.},
    VOLUME = {367},
      YEAR = {2015},
    NUMBER = {7},
     PAGES = {4983--5014},
      ISSN = {0002-9947},
       DOI = {10.1090/S0002-9947-2015-06197-2},
       URL = {https://doi-org.ucc.idm.oclc.org/10.1090/S0002-9947-2015-06197-2},
}

@article{Sogge,
    AUTHOR = {Sogge, Christopher D.},
     TITLE = {On the convergence of {R}iesz means on compact manifolds},
   JOURNAL = {Ann. of Math. (2)},
    VOLUME = {126},
      YEAR = {1987},
    NUMBER = {2},
     PAGES = {439--447},
      ISSN = {0003-486X},
       DOI = {10.2307/1971356},
       URL = {https://doi.org/10.2307/1971356},
}

@book{Stein,
   author={Stein, Elias M.},
   title={Harmonic analysis: real-variable methods, orthogonality, and
   oscillatory integrals},
   series={Princeton Mathematical Series},
   volume={43},
   note={With the assistance of Timothy S. Murphy;
   Monographs in Harmonic Analysis, III},
   publisher={Princeton University Press, Princeton, NJ},
   date={1993},
   pages={xiv+695},
   isbn={0-691-03216-5},
}

@article{Steinerberger2,
author = {Sagiv, Amir}, 
author = {Steinerberger, Stefan},
year = {2019},
pages = {},
title = {{Transport and Interface: an 
Uncertainty Principle for the Wasserstein 
distance}},
note = {\url{http://arxiv.org/abs/1905.07450}},
}

@article{Steinerberger0,
author = {Steinerberger, Stefan},
title = {A metric Sturm-Liouville theory in two dimensions},
journal = {Calc. Var. Partial Differential Equations},
number = {1},
year = {2020},
pages = {Paper 12},
}

@article{Steinerberger3,
author = {Steinerberger, Stefan},
year = {2017},
pages = {},
title = {{Oscillatory functions vanish on 
a large set}},
note = {\url{http://arxiv.org/abs/1708.05373}},
}

@article{Steinerberger1,
author = {Steinerberger, Stefan},
year = {2018},
pages = {},
title = {{Wasserstein Distance, Fourier 
Series and Applications}},
note = {\url{http://arxiv.org/abs/1803.08011}},
}

@book{Villani, 
author = {Villani, C\'edric},
title = {Optimal Transport, Old and New}, 
journal= {Grundlehren Math. Wiss. Springer}, 
volume = {Grundlehren Math. Wiss. 338. Springer},
edition = {1},
year = {2008}
}
\end{bibpropia}
\end{document}